\documentclass[11pt]{article}
\usepackage[a4paper,bindingoffset=0.in,left=2.5cm,right=2.5cm,top=2cm,bottom=3cm]{geometry}

\usepackage{mathtools,amsmath,amssymb,amsthm}
\usepackage{comment,xcolor}
\usepackage{multirow}
\usepackage{makecell}
\usepackage{array,booktabs}
\usepackage{color, colortbl}
\usepackage{kbordermatrix}
\usepackage{cancel} 
\usepackage{tabstackengine}[2016-11-30]
\usepackage{xcolor}
\usepackage[labelfont=bf]{caption}
\usepackage{algorithm}
\usepackage{algpseudocode}
\usepackage{siunitx}
\usepackage[affil-it]{authblk}
\usepackage{tikz}
\usepackage{blkarray}
\usetikzlibrary{arrows.meta}
\usetikzlibrary{positioning}

\usepackage{cite}

\newtheorem{theorem}{Theorem}[section]
\newtheorem{proposition}[theorem]{Proposition}

\newtheorem{lemma}[theorem]{Lemma}

\newtheorem{corollary}[theorem]{Corollary}

\theoremstyle{definition} 

\newtheorem{definition}[theorem]{Definition}

\newtheorem{example}[theorem]{Example}
\renewenvironment{proof}{{\noindent\bfseries Proof.}}{\qed} 

\usepackage[colorlinks=true,allcolors=blue!75!black,backref=page]{hyperref}

\newcommand{\rnk}{\operatorname{rank}}
\newcommand{\supp}{\operatorname{supp}}
\renewcommand{\vec}{\operatorname{vec}}
\newcommand{\gr}{\operatorname{Gr}}

\newcommand{\ZZ}{\mathbb{Z}}
\newcommand{\bb}{\mathbf{b}}
\newcommand{\be}{\mathbf{e}}
\newcommand{\bu}{\mathbf{u}}
\newcommand{\bv}{\mathbf{v}}
\newcommand{\bx}{\mathbf{x}}
\newcommand{\by}{\mathbf{y}}

\title{Markov and lattice bases for Forman-Ricci curvature of graphs}

\author[1]{Jane Ivy Coons\thanks{Contact: jcoons@wpi.edu}}
\author[2,3]{Giulio Zucal\thanks{Contact: zucal@mpi-cbg.de}}

\affil[1]{Worcester Polytechnic Institute, Massachusetts, USA}

\affil[2]{Max Planck Institute of Molecular Cell Biology and Genetics, Dresden, Germany}
\affil[3]{Center for Systems Biology Dresden, Germany}
\date{\today}

\begin{document}
\maketitle

\begin{abstract}

Discrete Forman-Ricci curvature is a quantity associated to each edge of a graph that describes its local geometry. It has proven to be a useful tool in network analysis in a variety of applications. Recent work by Roost et al.\ (2024) proposed the use of Markov bases to sample from the space of graphs with prescribed vertex degrees and curvatures. In the present work, we further develop the algebraic and combinatorial theory of these Markov bases. We show that the degree of an indispensable Markov move grows at least quadratically in the maximum degree of the graph. In light of this result, a compact description of all Markov basis elements seems unattainable at present. Instead, we find a lattice basis for this problem using only degree three Markov moves, which allows us to employ recently-developed reinforcement learning methods for finding Markov moves that can be applied to a specific graph.

\end{abstract}

\section{Introduction}\label{section: introduction}

Over the past several decades, the study of networks has become central across disciplines: from neurobiology and sociology to computer science and power grid optimization. Networks represent interactions between multiple agents and when only pairwise interactions are considered they can be naturally represented as graphs. While many traditional approaches to networks are problem-specific, recent research has shifted toward finding a unified language to transfer techniques between different disciplines. A powerful strategy in this direction is adapting well-established concepts from mathematics to networks. In particular, ideas from geometry have proven highly successful at capturing structural information of graphs.

A key geometric quantity is Ricci curvature, which measures how a Riemannian manifold deviates from Euclidean space. It has been generalized to many different contexts beyond Riemannian manifolds with graphs as important examples. In this setting, the properties characterizing Ricci curvature on manifolds are not equivalent anymore and for this reason we are equipped with a plethora of discrete Ricci curvatures each capturing different geometric features of the discrete structure.

The main focus of this article will be Forman-Ricci curvature \cite{forman2003bochner}. Originally introduced for CW-complexes via an analogy with the Bochner–Weitzenböck formula, its specialization to graphs yields a simple edge-based expression depending solely on the degrees of the connecting vertices. However, many other notions of discrete curvatures exist. 
For example, Ollivier-Ricci curvature \cite{ollivier2007ricci, ollivier2010survey} is based on optimal transport and neighbourhood overlap (and has been further refined by Lin, Lu and Yau \cite{OllivTheor2Lin}), and Bakry-Emery curvature \cite{bakry2013analysis} was introduced via a discrete analogue of the Bochner identity. We refer the reader to \cite{Devriendt_2022,devriendt2023graph,erbar_2012_ricci,steinerberger_2023_curvature} for further notions of discrete curvatures on graphs.

The study of  discrete Ricci curvature has attracted considerable attention recently. In particular, all the curvatures described above have been shown to enjoy similar properties of their smooth counterpart  \cite{OllivTheor1Lin, OllivTheor2Lin, OllivTheor3Bauer, OllivTheor4Jost, OllivTheor5Loisel, forman2003bochner, liu_2018_bakry, devriendt2023graph, steinerberger_2023_curvature}
and some of these quantities have been shown to play a key role in the context of Markov chain convergence \cite{salez2023cutoff, salez2024spectral, salez2025modernaspectsmarkovchains}. While we focus on the discrete curvature of graphs in the present paper, recent work has also addressed discrete curvature of the $2$-skeleta of polytopes   \cite{deloera2025discretecurvaturesconvexpolytopes}.

Beyond the theoretical interest, discrete curvatures have been also successfully applied to many scientific domains. These include analysis of biological and protein interaction networks \cite{OllivAppcancNet, Eidi2020EdgebasedAO} and detection of instabilities in financial networks \cite{CurvMarketInstab}. Moreover, discrete curvatures have been employed to mitigate oversquashing phenomena in graph neural networks \cite{topping_2022_understanding, fesser2024mitigating} and in community detection \cite{ComDet, fesser_2023_augmentations, tian2025jmlr}.

A recent line of work has addressed the inverse problem of constructing graphs with prescribed discrete curvature sequences. In particular, the authors of \cite{roost2024exploring} studied the problem of exact reconstruction of graphs from a given target Forman–Ricci curvature. Combining the algebraic theory of Markov bases together with characterizations of admissible joint degree sequences, they transformed this into a combinatorial problem. They proved that any two graphs sharing the same degree and curvature sequences can be transformed into one another using a finite number of rewiring moves, which can in theory be computed using computational algebraic geometry methods. These consist of simple moves (transpositions) and more complex changes dictated by the Markov basis (Markov moves).
A major computational bottleneck identified in \cite{roost2024exploring} is the explicit computation of these Markov bases. Using the computer algebra package 4ti2 within Macaulay2, Markov bases were obtained only for graphs with relatively small maximum node degree (up to degree 8), as computations for larger degrees failed to terminate within reasonable time. This limitation is not specific to the problem at hand: computing Markov bases is well known to be computationally expensive in general \cite{de2006markov}. Nevertheless, it is well known that specialized structural insights can substantially reduce complexity for particular classes of toric ideals and fibers.

Motivated by these challenges, the present work develops a deeper algebraic and combinatorial understanding of Markov bases associated with graphs of prescribed Forman–Ricci curvature and degree sequences. 
We show that the maximum degree of an indispensable Markov move grows at least quadratically in $\Delta$.
This result may be interpreted as a fundamental complexity limitation, demonstrating that the minimal algebraic complexity of Markov bases necessarily diverges quickly as $\Delta$ tends to infinity.

Despite this inherent complexity, we derive explicit formulas for simpler lattice bases, which can be used as building blocks for practical sampling algorithms. Furthermore, given the computational difficulty of symbolic Markov basis construction, we complement our theoretical analysis with a machine learning approach. Specifically, we employ the \emph{actor–critic fiber sampler} \cite{ActorCriticMarkovBases_Petrovic}, a recently-developed reinforcement learning algorithm designed to learn effective sampling moves within constrained fibers. We demonstrate that this method performs well for graphs with prescribed curvature and degree sequences of moderate size.

\subsection*{Organization}

 The remainder of this manuscript is structured as follows. In Section~\ref{Section2Defin}, we establish the notation and we provide  fundamental definitions regarding graphs and discrete curvature and the necessary algebraic background. In Sections~\ref{sec:DegreeDiverges} and~\ref{Sec:latt_basis}, we present our main theoretical results concerning Markov and lattice bases. In Section~\ref{Section:ReinforcemengtLearning}, we apply the actor-critic fiber sampler algorithm as a more computationally feasible approach to exploring fibers of specific graphs. 
\section{Definitions and notation}\label{Section2Defin}

\noindent We briefly recall some foundational concepts from graph theory and fix our notation. We introduce the concepts of Forman-Ricci curvatures and joint degree matrix of a graph. Finally we introduce Markov bases and describe recent work on how they can be used to explore the space of joint degree matrices of graphs with a fixed Forman-Ricci curvature \cite{roost2024exploring}. Since the fibers that we wish to sample from are constrained by upper bounds on the entries of a joint degree matrix, the Markov bases that we seek are in fact \emph{Graver bases}. We conclude the section with background on Graver bases and primitive binomials.

\begin{definition}
A \emph{simple graph} is a pair $G=(V,E)$, where $V=V(G)=\{v_1,\ldots,v_N\}$ is a finite set of vertices and $E=E(G) \subset 
\binom{V}{2}$ is a set of edges.
\end{definition}

Throughout, we fix a simple graph $G=(V,E)$.  
The \emph{degree} of a vertex $v\in V$, denoted by $\deg_G(v)$ or simply $\deg(v)$, is the number of edges incident to $v$. We assume that all vertices satisfy $\deg(v)\geq 1$, i.e., the graph has no isolated vertices.  If two vertices $v,w\in V$ are connected by an edge, we write $v\sim w$ (equivalently $w\sim v$), and denote the corresponding edge by $\{v,w\}$ or simply $vw$.

The \emph{maximum degree} of a graph $G = (V,E)$ is the maximum over all vertices $v \in V$ of $\deg_G(v)$.
Let $G=(V,E)$ be a graph with maximum degree $\Delta$.  
For each $a\in\{1,\ldots,\Delta\}$ we define
\[
V_a := \{v\in V : \deg(v)=a\}
\]
so that the cardinality of this set is the number of vertices of degree $a$.

\begin{definition}
The \emph{degree frequencies} of $G$ are given by the $\Delta$-tuple
\[
(\lvert V_1\rvert,\lvert V_2\rvert,\ldots,\lvert V_\Delta\rvert),
\]
whose $a$-th entry counts the number of vertices of degree $a$.
\end{definition}

\subsection{Forman--Ricci curvatures and joint degree matrices}\label{section: background, FR and JDM}
As discussed in the introduction, we investigate the structure of graphs using discrete Forman--Ricci curvature at each of its edges. Originally defined for general CW complexes \cite{forman2003bochner}, this curvature admits a particularly simple expression for graphs in terms of the degrees of the two vertices in an edge.

\begin{definition}
Let $G$ be a graph and let $e=\{u,v\}\in E(G)$.  
The \emph{Forman--Ricci curvature} of the edge $e$ is defined by
\begin{equation}\label{eq: definition Forman curvature}
F(e) = 4 - \deg(u) - \deg(v).
\end{equation}
\end{definition}
The graph in Figure \ref{fig: example FR} has each of its edges labeled with its Forman-Ricci curvature. This graph will serve as a running example throughout the paper. Note that the curvature of an edge $uv$ is determined by $\deg(u) + \deg(v)$, and that we must have $2 \leq \deg(u) + \deg(v) \leq 2\Delta$.
For $\kappa\in\{2,\ldots,2\Delta\}$ we define
\[
E_\kappa := \{uv\in E(G): \deg(u)+\deg(v)=\kappa\}.
\]
Equivalently, $E_\kappa$ is the set of edges with Forman--Ricci curvature $4-\kappa$.

\begin{figure}[h]
\centering
\begin{tikzpicture}[
    vertex/.style={circle, draw, thick, inner sep=2pt},
    edge/.style={thick}
]

\node[vertex] (a1) at (0,0) {};
\node[vertex] (a2) at (0,2) {};
\node[vertex] (a3) at (2,0) {};
\node[vertex] (a4) at (2,2) {};
\node[vertex] (a5) at (4,0) {};
\node[vertex] (a6) at (4,2) {};
\node[vertex] (a7) at (6,0) {};
\node[vertex] (a8) at (6,2) {};

\node at (0.65,1) {-2};
\node at (3.35,1) {-3};
\node at (5.35,1) {0};

\draw[edge] (a1) --node[above]{-1} (a3);
\draw[edge] (a1) -- (a4);
\draw[edge] (a2) --node[above]{-1} (a4);
\draw[edge] (a3) --node[left]{-3} (a4);
\draw[edge] (a3) --node[above]{-2} (a5);
\draw[edge] (a4) -- (a5);
\draw[edge] (a5) --node[right]{-1} (a6);
\draw[edge] (a6) -- (a7);
\draw[edge] (a7) --node[right]{1} (a8);
\end{tikzpicture}
\caption{Forman--Ricci curvature values $F(e)$ on the edges of a graph.}
\label{fig: example FR}
\end{figure}
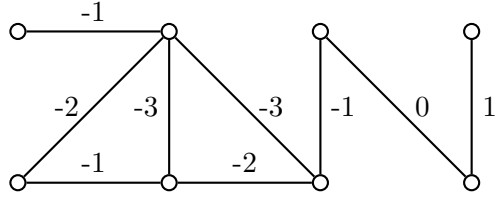

\begin{definition}
The \emph{curvature frequency sequence} of $G$ is the $(2\Delta-1)$-tuple
\[
(\lvert E_2\rvert,\ldots,\lvert E_{2\Delta}\rvert),
\]
where $\lvert E_\kappa\rvert$ denotes the number of edges with Forman--Ricci curvature $4-\kappa$.
\end{definition}

\begin{example}\label{example: degree and curvature frequencies}
Consider our running example graph depicted in Figure~\ref{fig: example FR}, which has maximum degree $\Delta = 4$. Its degree frequency sequence is $(2,3,2,1)$; that is, it has two vertices of degree $1$, three vertices of degree $2$, two vertices of degree $3$ and one vertex of degree $4$. Its curvature frequency sequence is $(0,1,1,3,2,2,0).$ Interpreting the fourth entry of this sequence corresponding to $E_5$, for example, we see that the graph has three edges $uv$ where $\deg(u) + \deg(v) = 5.$

\end{example}

Since the Forman--Ricci curvature at a given edge depends only on the degrees of its two endpoints, we associate to a graph the following matrix that encodes how vertex degrees are paired according to edges.

\begin{definition}
The \emph{joint degree matrix} (JDM) of a graph $G$ with maximum degree $\Delta$ is the symmetric
$\Delta\times\Delta$ matrix $J=(J_{ab})$ with entries
\[
J_{ab}
=\left\lvert
\bigl\{
uv\in E(G) : \{\deg(u),\deg(v)\}=\{a,b\}
\bigr\}
\right\rvert\]
That is, the $ab$ entry of $J$ is the number of edges in $G$ whose two endpoints have degrees $a$ and $b$.
\end{definition}

The JDM provides a compact representation of several global graph statistics.  
In particular, both degree and curvature frequencies can be recovered directly from $J$.
Indeed, the number of vertices of degree $a$ is obtained from the $a$-th row of the JDM as
\begin{equation}\label{eq: number of vertices from JDM}
|V_a|=\frac{1}{a}\Bigl(\sum_{b=1}^{\Delta} J_{ab}+J_{aa}\Bigr).
\end{equation}
Moreover, the number of edges with Forman--Ricci curvature $4-\kappa$ is given by the sum
\[
|E_\kappa|=\sum_{\substack{a\leq b\\ a+b=\kappa}} J_{ab}.
\]

\begin{example}\label{ex: jdm}
    The joint degree matrix of the graph pictured in Figure \ref{fig: example FR} is
    \[
    J = \begin{pmatrix}
        0 & 1 & 0 & 1 \\
        1 & 1 & 2 & 1 \\
        0 & 2 & 1 & 2 \\
        1 & 1 & 2 & 0
    \end{pmatrix}.
    \]
    Note, for example, that the number of vertices of degree $3$ is indeed $2 = \frac{1}{3} \big( 0 + 2 + 2\cdot 1 + 2\big)$. Furthermore the number of edges with curvature $-1 = 4 -5$ is
    \[
    |E_5| = \sum_{\substack{a \leq b \\ a + b = 5}} J_{ab} = J_{14} + J_{23} = 1 + 2 = 3. 
    \]
\end{example}

A complete characterization of matrices that arise as JDMs of simple graphs was obtained
by Stanton and Pinar~\cite{stanton_2012_constructing}. They establish that the JDMs for graphs with a fixed maximal degree are the lattice points in a certain semialgebraic set for which the expression given in Equation \ref{eq: number of vertices from JDM} is an integer.

\begin{theorem}[{\cite[Thm.~4.2]{stanton_2012_constructing}}]
\label{thm: characterization of JDMs}
A symmetric matrix $J\in\mathbb{N}^{\Delta\times\Delta}$ is the joint degree matrix of a simple
graph if and only if, for all distinct $a,b\in\{1,\dots,\Delta\}$, the following conditions hold:
\begin{itemize}
    \item[(i)] $\displaystyle |V_a|=\frac{1}{a}\Bigl(\sum_{c=1}^{\Delta}J_{ac}+J_{aa}\Bigr)$ is an integer;
    \item[(ii)] $J_{ab}\leq |V_a|\cdot |V_b|$;
    \item[(iii)] $J_{aa}\leq \binom{|V_a|}{2}$.
\end{itemize}
\end{theorem}

The following edge-rewiring move, known as a \emph{transposition}, preserves the JDM of a graph.
Let
\[
\mathcal{E}=\bigl\{ux, vy\bigr\}\subseteq E(G)
\]
be two edges such that vertex $u$ and vertex $v$ have the same degree and define
\[
\mathcal{E}'=\bigl\{uy,vx\bigr\}.
\]
Suppose further that neither $uy$ nor $vx$ belongs to $E(G)$.
The \emph{transposition} of the edge pair $\mathcal{E}$  (with respect to vertices $u$ and $v$) is the graph $G'$ obtained by replacing
$\mathcal{E}$ with $\mathcal{E}'$, that is,
\[
V(G')=V(G),
\qquad
E(G')=(E(G)\setminus\mathcal{E})\cup\mathcal{E}'.
\]
An example of a transposition is shown in Figure~\ref{fig: example transposition}. If we remove the assumption that $\{u,y\},\{x,v\} \not\in E(G)$, then performing a transposition can result in a multigraph.

\begin{figure}
   \centering
    \begin{tikzpicture}[scale=.95,
    vertex/.style={circle, draw, thick, inner sep=2pt},
    edge/.style={ultra thick},
    special_vertex/.style={circle, draw=blue, thick, fill=blue, inner sep=2pt},
    special_edge/.style={ultra thick, purple}
]

\node[special_vertex] (a1) at (0,0) {};
\node[vertex] (a2) at (0,2) {};
\node[vertex] (a3) at (2,0) {};
\node[special_vertex] (a4) at (2,2) {};
\node[special_vertex] (a5) at (4,0) {};
\node[special_vertex] (a6) at (4,2) {};
\node[vertex] (a7) at (6,0) {};
\node[vertex] (a8) at (6,2) {};

\node[below] at (0,-0.1) {v};
\node[above] at (2, 2.1) {y};
\node[above] at (4, 2.1) {u};
\node[below] at (4, -0.1) {x};

\draw[edge] (a1) -- (a3);
\draw[special_edge] (a1) -- (a4);
\draw[edge] (a2) -- (a4);
\draw[edge] (a3) -- (a4);
\draw[edge] (a3) -- (a5);
\draw[edge] (a4) -- (a5);
\draw[special_edge] (a5) -- (a6);
\draw[edge] (a6) -- (a7);
\draw[edge] (a7) -- (a8);

\draw[ultra thick, ->] (7,1) -- (9,1);

\node[special_vertex] (b1) at (10,0) {};
\node[vertex] (b2) at (10,2) {};
\node[vertex] (b3) at (12,0) {};
\node[special_vertex] (b4) at (12,2) {};
\node[special_vertex] (b5) at (14,0) {};
\node[special_vertex] (b6) at (14,2) {};
\node[vertex] (b7) at (16,0) {};
\node[vertex] (b8) at (16,2) {};

\node[below] at (10,-0.1) {v};
\node[above] at (12, 2.1) {y};
\node[above] at (14, 2.1) {u};
\node[below] at (14, -0.1) {x};

\draw[edge] (b1) -- (b3);
\draw[special_edge] (b1)  to[bend right=25] (b5);
\draw[edge] (b2) -- (b4);
\draw[edge] (b3) -- (b4);
\draw[edge] (b3) -- (b5);
\draw[edge] (b4) -- (b5);
\draw[special_edge] (b4) -- (b6);
\draw[edge] (b6) -- (b7);
\draw[edge] (b7) -- (b8);

\end{tikzpicture}
    \caption{A transposition replacing edges $\{u,x\}$ and $\{v,y\}$ with $\{u,y\}$ and $\{v,x\}$.}
    \label{fig: example transposition}
\end{figure}
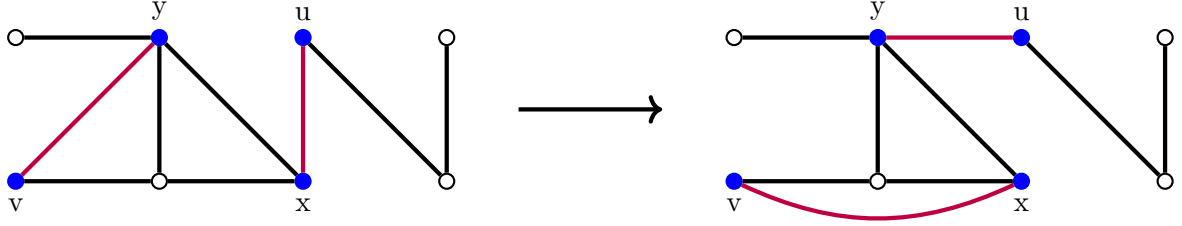

By construction, transpositions leave the JDM invariant.  
Moreover, the set of all graphs sharing the
same JDM is connected under transpositions
\cite{diaconis_2001_statistical, stanton_2012_constructing}.  
In particular, the space of simple graphs with a fixed JDM is connected via transpositions,
with all intermediate graphs remaining simple
\cite[Thm.~5.2]{stanton_2012_constructing}.
As a consequence, \cite{stanton_2012_constructing} also provides an explicit algorithm for
constructing a simple graph from a prescribed joint degree matrix.

Following the framework set out in \cite[Section~4]{roost2024exploring}, our goal is to explore the set of all JDMs of graphs with fixed degree and curvature frequency sequences using Markov bases. In the following section, we define Markov bases and describe how they can be applied to this problem.

\bigskip
\subsection{Markov bases}

The theory of Markov bases was introduced by Diaconis and Sturmfels
\cite{MarkovBases} in the context of sampling contingency tables, that is,
nonnegative integer matrices with fixed marginals. In general, Markov bases are used to sample non-negative integer vectors in a fiber of an integer matrix.
A key insight of this theory is that such sampling problems can be addressed
using tools from commutative algebra.
We briefly recall the main definitions and results before applying them to
the joint degree matrix setting.

\begin{definition}
Let $A\in\mathbb{Z}^{d\times n}$ be an integer matrix.
A finite set
$\mathcal{M}=\{\bb_1,\dots,\bb_M\}\subseteq\ker_{\mathbb{Z}}(A)$ forms a \emph{Markov basis} for $A$ if for any pair
$\bu,\bv\in\mathbb{N}^n$ with $A\bu=A\bv$, there exist $i_1,\dots,i_L \in [M]$ and $s_1,\dots,s_L \in \{-1,1\}$ such that
\begin{itemize}
    \item[(i)] $\bu+\sum_{j=1}^L s_j\bb_{i_j}=\bv$;
    \item[(ii)] $\bu+\sum_{j=1}^\ell s_j\bb_{i_j}\geq \mathbf{0}$ for all $1<\ell<L$.
\end{itemize}
\end{definition}

For $\bu\in\mathbb{N}^n$, the set
\[
\mathcal{F}_A(\bu)=\{\bv\in\mathbb{N}^n: A\bv=A\bu\}
\]
is called the \emph{fiber} of $\bu$. The elements of a Markov basis are often called \emph{Markov moves}. For any two elements of the same fiber of $A$, the moves in a Markov basis allow us to ``walk" from one element of the fiber to another while always remaining in the fiber. In fact, a Markov basis induces an ergodic Markov chain on each fiber by allowing transitions
$\bu\mapsto \bu+\bb$ whenever $\bb\in\mathcal{M}$ and $\bu+\bb\geq 0$. A Markov basis for $A$ is called \emph{minimal} if no proper subset of it also forms a Markov basis for $A$.

The connection between Markov bases and commutative algebra is established by the
Fundamental Theorem of Markov Bases.
For background material we refer to
\cite{aoki2012markov, MarkovBases, almendra-hernandez_2024_markov}.
Let $R=\mathbb{C}[x_1,\dots,x_n]$.
For $\bu\in\mathbb{N}^n$, write the monomial $\mathbf{x}^\bu=x_1^{u_1}\cdots x_n^{u_n}$.
For $\bu\in\mathbb{Z}^n$, define $\bu^+,\bu^-\in\mathbb{N}^n$ by
$(\bu^+)_i=\max\{\bu_i,0\}$ and $\bu^-=\bu^+-\bu$. In this way,, $\bu^+$ and $\bu^-$ are nonnegative integer vectors with disjoint support. We call $\bu^+$ the \emph{positive support} of $\bu$ and $\bu^-$ its \emph{negative support}.

\begin{theorem}[Fundamental Theorem of Markov Bases {\cite{MarkovBases}}]
\label{thm: fundamental theorem of Markov bases}
Let $A\in\mathbb{Z}^{d\times n}$.
A set $\mathcal{M}\subseteq\ker_{\mathbb{Z}}(A)$ is a Markov basis for $A$ if and only if the
binomials
$\{\mathbf{x}^{\bb^+}-\mathbf{x}^{\bb^-}:\bb\in\mathcal{M}\}$
generate the ideal $\langle \mathbf{x}^{\bb^+}-\mathbf{x}^{\bb^-}:\bb\in\ker_{\mathbb{Z}}(A)\rangle \subset R.$
\end{theorem}

This theorem shows that Markov bases correspond exactly to generating sets of the
\emph{toric ideal} defined by $A$, which we denote $I_A$.
In particular, by the Hilbert Basis Theorem, a finite Markov basis always exists. We will often abuse terminology and call a generating set for a given toric ideal a Markov basis for that ideal; a true Markov basis for the underlying integer matrix can then be obtained from the exponent vectors of the binomials in such a generating set.

\begin{example}
    Consider the matrix
    \[
    A = \begin{pmatrix}
        1 & 1 & 1 & 0 & 0 & 0 \\
        0 & 0 & 0 & 1 & 1 & 1 \\
        1 & 0 & 0 & 1 & 0 & 0 \\
        0 & 1 & 0 & 0 & 1 & 0 \\
        0 & 0 & 1 & 0 & 0 & 1
    \end{pmatrix} \in \ZZ^{5 \times 6}.
    \]
    Its corresponding toric ideal, $I_A = \langle \bx^{\bb^+} - \bx^{\bb^-} \mid \bb^+ - \bb^- \in \ker_\ZZ(A) \rangle $, is the ideal of a Segre embedding. We can compute a generating set for this toric ideal using the computer algebra software Macaulay2 \cite{M2} and we find that
    \[
    I_A = \langle x_1 x_5 - x_2 x_4, x_1 x_6 - x_3 x_4, x_2 x_6 - x_3 x_5 \rangle.
    \]
    Note that each of these binomials is a $2 \times 2$ minor of the generic integer matrix with entries $x_1, \dots, x_6$ read across rows. Recording the exponent vectors of these three binomials and applying the Fundamental Theorem of Markov Bases, we see that the set
    \[
    \mathcal{M}_A = \Big\{ \begin{pmatrix}
        1 \\ -1 \\ 0 \\ -1 \\ 1 \\ 0 
    \end{pmatrix}, \begin{pmatrix}
        1 \\ 0 \\ -1 \\ -1 \\ 0 \\ 1
    \end{pmatrix}, \begin{pmatrix}
        0 \\ 1 \\ -1 \\ 0 \\ -1 \\ 1
    \end{pmatrix}\Big\}
    \]
    forms a Markov basis for $A$. Denote the three elements of $\mathcal{M}$ by $\bb_1, \bb_2, \bb_3$ in their given order. Let $\bu = (1,0,3,4,1,2)^T$ and $\bv = (2,1,1,3,0,4)^T$. Note that $\bu$ and $\bv$ belong to the same fiber of $A$ as $A\bu = A\bv$. We have $\bv = \bu + \bb_3 + \bb_2$ and we note that $\bu + \bb_3$ is nonnegative. 
\end{example}

\subsection{Markov Bases for Joint Degree Matrices and Lawrence Liftings}

Following \cite{roost2024exploring}, we will apply the theory of Markov bases to the problem of exploring the set of all joint degree matrices with fixed degree and curvature sequences. From \eqref{eq: number of vertices from JDM}, we see that two JDMs $J$ and $K$ have the same degree frequencies if for all $a = 1,\dots, \Delta$,
\[
J_{aa} + \sum_{b=1}^\Delta J_{ab} = K_{aa} + \sum_{b=1}^\Delta K_{ab}.
\]
Similarly, they have the same curvature frequencies if for each $\kappa = 2,\dots, 2\Delta$, 
\[
\sum_{\substack{a \leq b \\ a + b = \kappa}} J_{ab} = \sum_{\substack{a \leq b \\ a + b = \kappa}} K_{ab}.
\]
We wish to encode this as $J$ and $K$ belonging to the same fiber of a certain matrix $B_\Delta$. First consider the vectorization of the unique entries of $J$, $\vec(J) = (J_{ab})_{a \leq b}$ listed in lexicographic order. Let $B_\Delta$ be a $(3\Delta -1) \times \binom{\Delta+1}{2}$ matrix whose columns are indexed by pairs $ij$ with $1 \leq i \leq j \leq \Delta$, also ordered lexicographically. For $1 \leq a \leq \Delta$, we have the entry
\[
B_\Delta(a, ij) = \begin{cases}
    2 & \text{ if } i = j = a, \\
    1 & \text{ if exactly one of $i$ or $j$ equals $a$,} \\
    0 & \text{ otherwise.}
\end{cases}
\]
For $2 \leq \kappa \leq 2\Delta$, we have the entry
\[
B_\Delta(\kappa + \Delta - 1, ij) = \begin{cases}
    1 & \text{ if } i + j = \kappa \text{ and} \\
    0 & \text{ otherwise.}
\end{cases}
\]
Then the JDMs $J$ and $K$ have the same degree and curvature frequencies if and only if $B_\Delta \vec(J) = B_\Delta \vec(K)$. In other words, this is the case if and only if $\vec(J)$ and $\vec(K)$ lie in the same fiber of $B_\Delta$.

However, running the Markov chain over a fiber of $B_\Delta$ can yield matrices which are not JDMs of any graph. Indeed, the matrices obtained in this way may not satisfy inequalities (ii) and (iii) of Theorem \ref{thm: characterization of JDMs}. To avoid this, we modify the problem by adding slack variables $S_{ab}$ for $1 \leq a \leq b \leq \Delta$. Fix a JDM $J$ for a graph with $|V_a|$ vertices of degree $a$ for each $a$. Then we require that $S_{ab} \geq 0$ for all $a$ and $b$, that $S_{aa} + J_{aa} = \binom{|V_a|}{2}$ for all $a$ and that $S_{ab} + J_{ab} = |V_a| |V_b|$ for all $a \neq b$. To add these constraints, we use the \emph{Lawrence lifting} of $B_\Delta$,
\begin{equation}\label{eq: lawrence lifting}
    \Lambda(B_\Delta) := \begin{pmatrix}
        B_\Delta & \mathbf{0}_{(3\Delta -1) \times \binom{\Delta+1}{2}} \\
        I_{\binom{\Delta+1}{2}} & I_{\binom{\Delta+1}{2}}
    \end{pmatrix}.
\end{equation}
In this matrix, the last $\binom{\Delta+1}{2}$ columns correspond to the slack variables $S_{ab}$. Then starting with a JDM $J$ and appropriate values of the slack variables $S$, the fiber of $\Lambda(B_\Delta)$ containing $(J,S)$ consists exactly of the JDMs of graphs with the same degree and curvature frequencies as $J$. The use of Lawrence liftings for sampling from bounded fibers was first described in \cite{rogantin2007markov} and further developed in \cite{rapallo2010markov}.

\begin{example}
    The matrix $B_5$ is a $14 \times 15$ integer matrix. Its columns are indexed by pairs $ij$ with $1 \leq i \leq j \leq \Delta$. Its rows are split into two blocks of size $\Delta$ and $2 \Delta -1$: the first (labeled in blue below) is indexed by degree frequencies for degrees $1,\dots, \Delta$ and the second (labeled in red) is indexed by curvature frequencies $\kappa = 2,\dots, 2\Delta$. This matrix is
    \[
    B_5 = \begin{blockarray}{*{16}{c}}
    & 11 & 12 & 13 & 14 & 15 & 22 & 23 & 24 & 25 & 33 & 34 & 35 & 44 & 45 & 55 \\
    \begin{block}{c(ccccccccccccccc)}
        {\color{blue} 1} & 2 & 1 & 1 & 1 & 1 & 0 & 0 & 0 & 0 & 0 & 0 & 0 & 0 & 0 & 0 \\
        {\color{blue} 2} & 0& 1& 0& 0& 0& 2& 1& 1& 1& 0& 0& 0& 0& 0& 0  \\
        {\color{blue} 3} & 0 & 0& 1& 0& 0& 0& 1& 0& 0& 2& 1& 1& 0& 0& 0\\
        {\color{blue} 4} & 0& 0& 0& 1& 0& 0& 0& 1& 0& 0& 1& 0& 2& 1& 0\\
        {\color{blue} 5} & 0& 0& 0& 0& 1& 0& 0& 0& 1& 0& 0& 1& 0& 1&2 \\
        \BAhline
        {\color{red} 2} & 1& 0& 0& 0& 0& 0& 0& 0& 0& 0& 0& 0& 0& 0& 0\\
        {\color{red} 3} & 0& 1& 0& 0& 0& 0& 0& 0&0 & 0& 0& 0& 0& 0& 0\\
        {\color{red} 4} & 0& 0& 1& 0& 0& 1& 0& 0& 0& 0& 0& 0& 0& 0& 0\\
        {\color{red} 5} & 0& 0& 0& 1& 0& 0& 1& 0& 0& 0& 0& 0& 0& 0& 0\\
        {\color{red} 6} & 0& 0& 0& 0& 1& 0& 0& 1& 0& 1& 0& 0& 0& 0& 0\\
        {\color{red} 7} & 0& 0& 0& 0& 0& 0& 0& 0& 1& 0& 1& 0& 0& 0& 0\\
        {\color{red} 8} & 0& 0& 0& 0& 0& 0& 0& 0& 0&0 & 0& 1& 1& 0& 0\\
        {\color{red} 9} & 0& 0& 0& 0& 0& 0& 0& 0& 0& 0& 0& 0& 0& 1& 0\\
        {\color{red} 10} & 0& 0& 0& 0& 0& 0& 0& 0& 0& 0& 0& 0& 0& 0& 1\\
    \end{block}
    \end{blockarray}.
    \]
    The Lawrence lifting of $B_5$ is the $29 \times 30$ matrix described in Equation \ref{eq: lawrence lifting}.
\end{example}

Given any integer matrix $A$, its Lawrence lifting $\Lambda(A)$ has a unique minimal Markov basis \cite[Theorem~7.1]{sturmfels1996grobner}. In order to describe it, we must first define the notion of a primitive Markov move.

\begin{definition}\label{def:primitive}
    Let $A$ be an integer matrix and let $I_A$ be its corresponding toric ideal,
    \[
    I_A = \langle \bx^{\bu^+} - \bx^{\bu^-} \mid \bu^+ - \bu^- \in \ker_\ZZ(A) \rangle.
    \]
    A binomial $\bx^{\bu^+} - \bx^{\bu^-} \in I_A$ is called \emph{primitive} if there does not exist a different nonzero binomial $\bx^{\bv^+} - \bx^{\bv^-} \in I_A$ such that $\bx^{\bv^+}$ divides $\bx^{\bu^+}$ and $\bx^{\bv^-}$ divides $\bx^{\bu^-}$.
\end{definition}

Reframing these moves as integer vectors, this says that a move $\bu^+ - \bu^-$ is primitive if there does not exist another nonzero move $\bv^+ - \bv^-$ such that $\bu^+ - \bv^+ \geq \mathbf{0}$ and $\bu^- - \bv^- \geq \mathbf{0}$. We note that this definition of primitivity differs from the standard definition of a primitive vector of an integer lattice. However, since we only use the notion described in  Definition \ref{def:primitive} in the present work, we refer to an integer vector as a primitive move for the matrix $A$ if its corresponding binomial is primitive in the Markov basis sense.

\begin{definition}
    The \emph{Graver basis} of a toric ideal $I_A$ is the set of all primitive binomials in $I_A$. We denote it by $\gr(I_A)$.
\end{definition}

 The following theorem classifies the unique minimal Markov basis of a matrix of Lawrence type.

\begin{theorem}[\cite{sturmfels1996grobner} Theorem~7.1]\label{ThmGraver}
Let $A$ be an integer matrix and let $\Lambda(A)$ be its Lawrence lifting. The toric ideal $I_{\Lambda(A)}$ has a unique minimal Markov basis which is equal to the Graver basis, $\gr(I_{\Lambda(A)})$. 
\end{theorem}

Let $A$ be a $d\times n$ integer matrix. Given a primitive binomial $\bx^{\bu^+} - \bx^{\bu^-} \in I_A$, we can find a primitive binomial in $I_{\Lambda(A)}$ via the following lifting operation. Let $I_{\Lambda(A)} \subset k[\bx, \by]$ where the variables $\bx$ correspond to the first $n$ columns of $\Lambda(A)$ and the variables $\by$ correspond to the last $n$ columns. The \emph{lift} of $\bx^{\bu^+} - \bx^{\bu^-}$ to $I_{\Lambda(A)}$ is the binomial $\bx^{\bu^+}\by^{\bu^-} - \bx^{\bu^-}\by^{\bu^+}$. On the level of integer vectors, this says that $\bu^+ - \bu^-$ is primitive for $A$ if and only if the move 
\begin{equation}\label{eq: move lifting}
\begin{pmatrix}
    \bu^+ \\
    \bu^-
\end{pmatrix} - \begin{pmatrix}
    \bu^- \\
    \bu^+
\end{pmatrix}
\end{equation}
is primitive for $\Lambda(A)$. It is not difficult to check that all binomials in $I_{\Lambda(A)}$ are lifts of binomials in $I_A$ and that a binomial in $I_A$ is primitive if and only if its lift is primitive for $I_{\Lambda(A)}$. Applying this to our setting, we see that in order to find the Markov basis for $\Lambda(B_\Delta)$, it suffices to compute the primitive moves of $B_\Delta$.

Since the vector of all ones belongs to the rowspan of $B_\Delta$, we see that if $\bu^+ - \bu^- \in \ker(B_\Delta)$, then the sum of the entries of $\bu^+$ is equal to the sum of the entries of $\bu^-$. In other words, the toric ideal $I_{B_\Delta}$ is homogeneous. In light of Theorem \ref{thm: fundamental theorem of Markov bases}, we define the \emph{degree} of the Markov move $\bu^+ - \bu^-$ to be the sum of the entries of $\bu^+$ (or equivalently $\bu^-$); this is equal to the total degree of the corresponding binomial in $I_{B_\Delta}$. For pair of (possibly non-distinct) $i,j \in \{1,\dots,\Delta\}$, let $\be_{ij}$ denote the $\Delta \times \Delta$ matrix whose $ij$ and $ji$ entries are $1$ and all of whose other entries are $0$. Equivalently, we may think of $\be_{ij}$ as a standard basis vector in $\mathbb{R}^{\binom{\Delta+1}{2}}$.

\begin{example}
    Consider the graph $G$ pictured in Figure \ref{fig: example FR}. We computed its joint degree matrix $J$ in Example \ref{ex: jdm} and would now like to apply a Markov move to it and interpret that move graphically. Since the maximum degree of $G$ is $4$, we consider the matrix $\Lambda(B_4)$. The kernel of $B_4$ is one-dimensional and spanned by
    \[
    \mathbf{m} = \be_{13} + \be_{23} + \be_{24} - \be_{14} - \be_{22} - \be_{33}.
    \]
    Thus the kernel of $\Lambda(B_4)$ is spanned by $\begin{bmatrix} \mathbf{m} \\ -\mathbf{m} \end{bmatrix}$. Viewing $\mathbf{m}$ as a symmetric matrix, we apply this move to the JDM of $G$ by
    \[
    J + \mathbf{m} = \begin{pmatrix}
        0 & 1 & 0 & 1 \\
        1 & 1 & 2 & 1 \\
        0 & 2 & 1 & 2 \\
        1 & 1 & 2 & 0
    \end{pmatrix} + \begin{pmatrix}
        0 & 0 & 1 & -1 \\
        0 & -1 & 1 & 1 \\
        1 & 1 & -1 & 0 \\
        -1 & 1 & 0 & 0
    \end{pmatrix} = \begin{pmatrix}
        0 & 1 & 1 & 0 \\
        1 & 0 & 3 & 2 \\
        1 & 3 & 0 & 2 \\
        0 & 2 & 2 & 0
    \end{pmatrix}.
    \]
    Note that in this case, the resulting matrix satisfies the conditions of Theorem \ref{thm: characterization of JDMs} and is the JDM of a graph. On the level of the graph $G$, applying this move corresponds to replacing three edges with endpoint degrees $\{1,4\}, \{2,2\}$ and $\{3,3\}$ with three edges of endpoint degrees $\{1,3\}, \{2,3\}$ and $\{2,4\}.$ Figure \ref{fig: markov move application} depicts a graph obtained from $G$ by performing this Markov move.
\end{example}

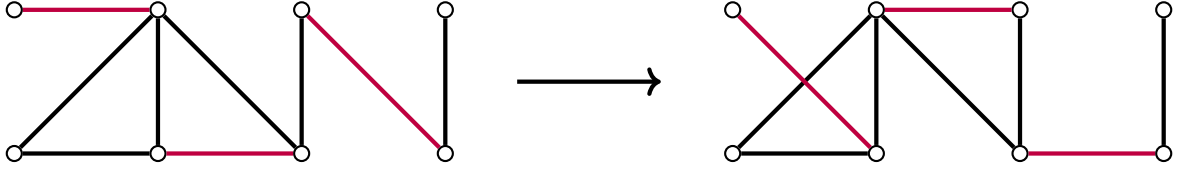
\begin{figure}
\centering
    \begin{tikzpicture}[scale=.95,
    vertex/.style={circle, draw, thick, inner sep=2pt},
    edge/.style={ultra thick},
    special_vertex/.style={circle, draw=blue, thick, fill=blue, inner sep=2pt},
    special_edge/.style={ultra thick, purple}
]

\node[vertex] (a1) at (0,0) {};
\node[vertex] (a2) at (0,2) {};
\node[vertex] (a3) at (2,0) {};
\node[vertex] (a4) at (2,2) {};
\node[vertex] (a5) at (4,0) {};
\node[vertex] (a6) at (4,2) {};
\node[vertex] (a7) at (6,0) {};
\node[vertex] (a8) at (6,2) {};

\draw[edge] (a1) -- (a3);
\draw[edge] (a1) -- (a4);
\draw[special_edge] (a2) -- (a4);
\draw[edge] (a3) -- (a4);
\draw[special_edge] (a3) -- (a5);
\draw[edge] (a4) -- (a5);
\draw[edge] (a5) -- (a6);
\draw[special_edge] (a6) -- (a7);
\draw[edge] (a7) -- (a8);

\draw[ultra thick, ->] (7,1) -- (9,1);

\node[vertex] (b1) at (10,0) {};
\node[vertex] (b2) at (10,2) {};
\node[vertex] (b3) at (12,0) {};
\node[vertex] (b4) at (12,2) {};
\node[vertex] (b5) at (14,0) {};
\node[vertex] (b6) at (14,2) {};
\node[vertex] (b7) at (16,0) {};
\node[vertex] (b8) at (16,2) {};

\draw[edge] (b1) -- (b3);
\draw[edge] (b1) -- (b4);
\draw[special_edge] (b2) -- (b3);
\draw[edge] (b3) -- (b4);
\draw[special_edge] (b5) -- (b7);
\draw[edge] (b4) -- (b5);
\draw[edge] (b5) -- (b6);
\draw[special_edge] (b6) -- (b4);
\draw[edge] (b7) -- (b8);
\end{tikzpicture}

\caption{The Markov move $\mathbf{m}$ applied to a graph.} \label{fig: markov move application}    
\end{figure}

When applying a Markov move to a given graph, it may be necessary to create a multigraph with parallel edges or self-loops. However, since the JDM of the multigraph created must satsify the conditions of Theorem \ref{thm: characterization of JDMs}, it is possible to apply transpositions to the resulting multigraph in order to obtain a (simple) graph.

Graver bases are powerful tools for solving integer linear programs \cite{schrijver1998theory}, as well as certain nonlinear integer programs. Specifically, applying Theorem \ref{ThmGraver} and the theory of Graver bases to our setting yields a method for solving optimization problems on JDMs constrained by a given degree and curvature sequence. A promising direction for future research is leveraging this framework to study extremal graphs with prescribed degree and curvature sequences.

\section{Degree growth of Markov moves}\label{sec:DegreeDiverges}

In this section, we give a family of primitive moves for $B_\Delta$ for even values of $\Delta$ whose degree grows on the order of $\Delta^2$. Since the unique minimal Markov basis for the Lawrence lifting $\Lambda(B_\Delta)$ consists of the lifts of Graver basis elements of $B_\Delta$, this construction shows that the degree of minimal Markov basis elements for $\Lambda(B_\Delta)$ grows quadratically with $\Delta$.

Let $\Delta \geq 6$ be even. We will define a Markov move for $B_\Delta$, $\bu_\Delta = \bu_\Delta^+ - \bu_\Delta^-$ where $\bu_\Delta^+$ and $\bu_\Delta^-$ are nonnegative integer vectors with disjoint support. In order to define $\bu_\Delta$, we require the following coefficients. For each $i$ between $3$ and $\Delta-2$, we define
\[
a_i = \begin{cases}
    \Delta - i - 1 & \text{ if $i$ odd, $i \leq \frac{\Delta}{2}$} \\
    i -2 & \text{ if $i$ even, $i \leq \frac{\Delta}{2}$} \\
    a_{\Delta - i + 1} & \text{ if } i > \frac{\Delta}{2} \\
\end{cases}
\]
Using these, we define
\begin{align}
    \bu_\Delta^+ &= \be_{2,2} + \sum_{i=3}^{\Delta-2} a_i \be_{i,i} + \be_{\Delta-1,\Delta-1} + (\Delta-3) \be_{1,\Delta} \text{ and } \label{eq:MoveDefinition}\\
    \bu_\Delta^- &= \be_{1,3} + \sum_{i=3}^{\Delta-2} a_i \be_{i-2,i+2} + \be_{\Delta-2,\Delta} + (\Delta-3)\be_{3, \Delta-2}. \nonumber
\end{align}
Finally, let $\bu_\Delta = \bu_\Delta^+ - \bu_\Delta^-.$ When $\Delta$ has been fixed, we simply denote $\bu_\Delta$ by $\bu$. 

\begin{example}
    Consider the case where $\Delta = 10$. The relevant coefficients $a_i$ are $a_3 = 6$, $a_4 = 2$, $a_5 = 4$, $a_6 = 4$, $a_7 = 2$ and $a_8 = 6$. The Markov move $\bu_{10}$ is $\bu_{10}^+ - \bu_{10}^-$ where
    \begin{align*}
        \bu_{10}^+ &= \be_{2,2} + 6\be_{3,3} + 2\be_{4,4} + 4\be_{5,5} + 4\be_{6,6} + 2\be_{7,7} + 6\be_{8,8} + \be_{9,9} + 7\be_{1,10}, \text{ and} \\
        \bu_{10}^- &= \be_{1,3} + 6\be_{1,5} + 2\be_{2,6} + 4\be_{3,7} + 4\be_{4,8} + 2\be_{5,9} + 6 \be_{6,10} + \be_{8,10} + 7 \be_{3,8}.
    \end{align*}
\end{example}

For any Markov move $\bv^+ - \bv^-$, the \emph{multiplicity} of an index $i$ is the number of times $i$ appears as an index in $\bv^+$ (or equivalently, $\bv^-$). In other words, it is $v^+_{ii}+\sum_{j=1}^\Delta v^+_{ij}$. Similarly, the multiplicity of the sum $\kappa$ is $\sum_{\substack{i + j = \kappa , \ i \leq j}} v^+_{ij}$.
In this section, we will show that $\bu$ is a primitive Markov move for $B_\Delta$. To achieve this, we will first show that $\bu$ is in the kernel of $B_\Delta$; in other words, that the indices and their sums that appear in $\bu^+$ also appear in $\bu^-$ with the same multiplicities. Then we show that any move whose support is contained in $\supp(\bu)$ must either be the $0$-vector or have support equal to \emph{all} of $\supp(\bu)$.

\begin{lemma}\label{lem:QuadtricMoveCorrectness}
    Let $\Delta \geq 6$ be even and let $\bu_\Delta = \bu_\Delta^+ - \bu_\Delta^-$ as defined in Equation \ref{eq:MoveDefinition}. Then $\bu_\Delta$ belongs to the kernel of $B_\Delta$.
\end{lemma}

\begin{proof}
    Since $\Delta$ is fixed, we omit the $\Delta$ subscript and write $\bu = \bu_\Delta$, $\bu^+ = \bu_\Delta^+$ and $\bu^- = \bu^-_\Delta$ to simplify notation.
    First note that the sums of indices that appear in $\bu^+$ and $\bu^-$ are exactly $\Delta+1$ and the even numbers $4, 6, \dots, 2\Delta-2$. They appear with the same multiplicity by construction. It remains to show that each index $1,\dots,\Delta$ appears with the same multiplicity in $\bu^+$ and $\bu^-$. Since $\bu$ is invariant under the permutation of indices which sends each  $i$ to $\Delta-i+1$, it suffices to show this is the case for $1 \leq i \leq \frac{\Delta}{2}$. We consider the special cases, $i =1 ,\dots, 4$ and the general case where $5 \leq i \leq \frac{\Delta}{2}$.

    The index $1$ appears in $\bu^+$ only in the pair $(1,\Delta)$ and the coefficient of $\be_{1,\Delta}$ is $\Delta-3$. It appears in $\bu^-$ in the pairs $(1,3)$, which has coefficient $1$, and $(1,5)$, which has coefficient $a_3 = \Delta - 4$. So the multiplicity of index $1$ in $\bu^-$ is $1 + \Delta - 4 = \Delta -3$, as needed.

    The index $2$ appears in $\bu^+$ only in the pair $(2,2)$ with coefficient $1$; hence its multiplicity in $\bu^+$ is $2$. We note here that each instance of $2$ in the pair $(2,2)$ counts separately towards its multiplicity. It appears in $\bu^-$ only in the pair $(2,6)$, which appears with coefficient $a_4 = 4-2=2$, as needed.

    The index $3$ appears in $\bu^+$ only in the pair $(3,3)$ with coefficient $a_3 = \Delta - 4$. So its multiplicity in $\bu^+$ is $2a_3 = 2\Delta - 8$. It appears in three pairs in $\bu^-$: the pair $(1,3)$ with coefficient $1$, the pair $(3,7)$ with coefficient $a_5 = \Delta - 6$ and the pair $(3,\Delta-2)$ with coefficient $\Delta - 3$. So the multiplicity of the index $3$ in $\bu^-$ is $1 + \Delta - 6 + \Delta - 3 = 2\Delta - 8$, as needed.

    The index $4$ appears in $\bu^+$ only in the pair $(4,4)$ with coefficient $a_4 = 2$. So its multiplicity in $\bu^+$ is $4$. It appears in $\bu^-$ only in the pair $(4,8)$ with coefficient $a_6 = 4$, and hence its multiplicity in $\bu^-$ is also $4$.

    Finally, let $i$ be such that $5 \leq i \leq \frac{\Delta}{2}$. The index $i$ appears in $\bu^+$ with multiplicity $2a_i$ and in $\bu^-$ with multiplicity $a_{i-2} + a_{i+2}$. We claim that for every such $i$, $2a_i = a_{i-2}+a_{i+2}$. There are several cases.

    First suppose $i \leq \frac{\Delta}{2} - 2$. If $i$ is odd, then $2a_i = 2 \Delta - 2i - 2$ and $a_{i-2} + a_{i+2} = \Delta - (i-2) - 1 + \Delta - (i+2) - 1 = 2\Delta - 2i - 2.$
    If $i$ is even, then $2a_i = 2i -4$ and $a_{i-2} + a_{i+2} = i - 4 + i = 2i-4.$

    Next, consider the case where $i = \frac{\Delta}{2} - 1$. If $\frac{\Delta}{2} - 1$ is odd, then 
    \begin{align*}
    2a_{\frac{\Delta}{2} - 1} &= 2 \Big( \Delta - \big(\frac{\Delta}{2} - 1\big) - 1 \Big) = \Delta \text{ and } \\
    a_{\frac{\Delta}{2} - 3} + a_{\frac{\Delta}{2} + 1} & = a_{\frac{\Delta}{2} - 3} + a_{\frac{\Delta}{2}} \\
    & = \Delta - \big( \frac{\Delta}{2} - 3 \big) - 1 + \frac{\Delta}{2} - 2 \\
    & = \Delta.
    \end{align*}
    Similarly if $\frac{\Delta}{2} - 1$ is even, then $2a_{\frac{\Delta}{2} - 1} = 2\big(\frac{\Delta}{2} - 3) = \Delta - 6$. In this case, $a_{\frac{\Delta}{2} - 3} + a_{\frac{\Delta}{2} + 1} = \frac{\Delta}{2} - 5 + \Delta - \frac{\Delta}{2} - 1 = \Delta - 6$, as needed.

    Finally we consider the case $i = \frac{\Delta}{2}$. If $\frac{\Delta}{2}$ is odd, then $2a_{\frac{\Delta}{2}} = 2(\frac{\Delta}{2}-1) = \Delta - 2$. We have
    \[a_{\frac{\Delta}{2}-2} + a_{\frac{\Delta}{2}+2} = a_{\frac{\Delta}{2}-2} + a_{\frac{\Delta}{2}-1} = \Delta - (\frac{\Delta}{2} - 2) - 1 + (\frac{\Delta}{2} -1) - 2 = \Delta - 2.\]
    If $\frac{\Delta}{2}$ is even, then $2a_{\frac{\Delta}{2}} = \Delta-4$ and $a_{\frac{\Delta}{2}-2} + a_{\frac{\Delta}{2}+2} = \frac{\Delta}{2} - 4 + \Delta - (\frac{\Delta}{2}-1) -1 = \Delta-4.$

    By the invariance of $\bu$ under the transformation of indices sending $i$ to $\Delta-i+1$, this exhausts all cases. Since the multiplicity of each index and each sum of pairs of indices appearing in $\bu^+$ are the same as those in $\bu^-$, we conclude that $\bu$ belongs to the kernel of $B_\Delta$.
\end{proof}

Now we must show that $\bu$ is a \emph{primitive} element of $\ker(B_\Delta)$. In other words, we must show that there is no other nonzero $\bv \in \ker(B_\Delta)$ such that $\bu^+ - \bv^+ \geq \mathbf{0}$ and $\bu^- - \bv^- \geq \mathbf{0}$. To achieve this, we first show that for any such $\bv$, we must have $\supp(\bv^+) = \supp(\bu^+)$ and $\supp(\bv^-) = \supp(\bu^-)$. 

\begin{lemma}\label{lem:MoveWithout22}
    Let $\bv \in \ker_\ZZ(B_\Delta)$ be such that $\supp(\bv^+) \subset \supp(\bu_\Delta^+)$ and $\supp(\bv^-) \subset \supp(\bu_\Delta^-)$. Suppose further that $(2,2)$ is not in the support of $\bv^+$. Then $\bv$ is the zero vector.
\end{lemma}

\begin{proof}
    We again omit the $\Delta$ subscript from $\bu, \bu^+$ and $\bu^-$ to streamline our notation.
    First note that the sums of indices in $\supp(\bu^+)$ are all distinct and the same is true of the indices in $\supp(\bu^-)$. Thus if any pair $(i,j) \in \supp(\bu^+)$ does not belong to $\supp(\bv^+)$, then the corresponding pair $(k, \ell) \in \supp(\bu^-)$ with $k + \ell = i +j$ does not belong to $\supp(\bv^-)$ and vice versa.
    
    Now suppose that $v^+_{22} = 0$. So by the above observation, $v^-_{13} = 0$ as well. Then there are no pairs in the support of $\bv^-$ that contain the index $2$; in particular, from the structure of $u^-$ given in Equation \ref{eq:MoveDefinition}, we see that $v^-_{26} = 0$ and $v^+_{44} = 0$. Hence there are also no pairs in the support of $\bv^-$ containing the index $4$, and we conclude that $v^-_{48} = v^+_{66} = 0$. Proceeding in this manner, we see that for all even $i \geq 4$, $v^+_{ii} = 0$ and $v^-_{i-2,i+2} = 0$. In particular, $v^+_{\Delta-2,\Delta-2} = 0$, so there are no pairs in $\supp(\bv^-)$ with the index $\Delta-2$. We conclude that $v^-_{\Delta-2,\Delta} = v^-_{3,\Delta-2} = 0$. Hence $v^+_{\Delta-1,\Delta-1} = v^+_{1,\Delta} = 0$. 

    Now we see that there are no pairs in $\supp(\bv^-)$ including the index $1$, so in particular, $v^-_{15} = 0$ and $v^+_{33} = 0$. This excludes $3$ from the indices of $\supp(\bv^-)$, so $v^-_{37} = v^+_{55} = 0$. Proceeding in the same way, we see that for all odd $i \geq 3$, $v^+_{ii} = v^-_{i-2,i+2} = 0$. So we have found that all of the possibly nonzero coordinates of $\bv^+$ and $\bv^-$ are in fact equal to $0$, so $\bv^+ = \bv^- = \mathbf{0}$.
\end{proof}

This lemma finally allows us to prove that $\bu_\Delta$ is a primitive Markov move.

\begin{theorem}\label{thm:QuadraticMovePrimativity}
    Let $\Delta \geq 6$ be even and let $\bu_\Delta = \bu_\Delta^+ - \bu_\Delta^-$ as defined in Equation \ref{eq:MoveDefinition}. Then $\bu_\Delta$ is a primitive Markov move for $B_\Delta$.
\end{theorem}

\begin{proof}
    First, Lemma \ref{lem:QuadtricMoveCorrectness} shows that $\bu_\Delta$ indeed belongs to $\ker(B_\Delta)$. To show that $\bu_\Delta$ is primitive, suppose that $\bv \in \ker_\ZZ(B_\Delta)$ such that $\bu_\Delta^+ - \bv^+ \geq \mathbf{0}$ and $\bu_\Delta^- - \bv^- \geq \mathbf{0}$.  Then $(\bu_\Delta^+ - \bv^+) - (\bu_\Delta^- -\bv^-)$ also belongs to $\ker_\ZZ(B_\Delta)$; moreover the vectors $\bu_\Delta^+ - \bv^+ $ and $ \bu_\Delta^- -\bv^-$ are nonnegative with disjoint support.

    If $v^+_{22} = 0$, then by Lemma \ref{lem:MoveWithout22}, $\bv$ is the zero vector. If $v^+_{22} = 1$, then the $(2,2)$ entry of the Markov move $(\bu_\Delta^+ - \bv^+) - (\bu_\Delta^- -\bv^-)$ is $0$. Hence by Lemma \ref{lem:MoveWithout22}, $(\bu_\Delta^+ - \bv^+) - (\bu_\Delta^- -\bv^-)$ is the zero vector, and since the supports of $\bu_\Delta^+$ and $\bu_\Delta^-$ are disjoint, we have $\bu_\Delta^+ = \bv^+$ and $\bu_\Delta^- = \bv^-$. Thus if $\bv \in \ker_\ZZ(B_\Delta)$ satisfies the conditions that $\bu_\Delta^+ - \bv^+ \geq \mathbf{0}$ and $\bu_\Delta^- - \bv^- \geq \mathbf{0}$, then $\bv$ is either the zero vector or $\bu_\Delta$. Hence $\bu_\Delta$ is primitive.
\end{proof}

Finally we consider how the degree of $\bu_\Delta$ grows with $\Delta$. Recall that the degree of the Markov move $\bu_\Delta^+ - \bu_\Delta^-$ is the coordinate sum of $\bu_\Delta^+$ (or equivalently, of $\bu_\Delta^-$). In the following theorem, we use the growth of the coefficients $a_i$ to show that the degree of $\bu_\Delta$ grows quadratically in $\Delta$. This means that the degree of elements of the unique minimal Markov basis for $\Lambda(B_\Delta)$ grows at least quadratically as well.

\begin{theorem}\label{thm:uDeltaDegree}
    For each integer $\Delta \geq 4$, there exists a primitive Markov move for $B_\Delta$ of degree $2\big(\lfloor \frac{\Delta}{2} \rfloor -1 \big)^2 + 1$.
\end{theorem}

\begin{proof}
    First, we consider the case where $\Delta$ is even. When $\Delta = 4$, the unique minimal Markov move for $A_4$ has the desired degree $3$; in fact, this move is the analogue of $\bu$ for $\Delta = 4$ since in this case, there are no coefficients $a_i$. Now consider the degree of $\bu_\Delta$ for even $\Delta \geq 6$. It is
    \begin{align*}
        2 + (\Delta - 3) + 2\sum^{\Delta-4}_{\substack{i=2 \\ i \text{ even}}} i &= \Delta - 1 + 4 \sum_{i=1}^{\frac{\Delta}{2}-2} i \\
        &= \Delta - 1 + 4 \binom{\frac{\Delta}{2} - 1}{2} \\
        &= \Delta - 1 + 2\Big(\frac{\Delta}{2}-1\Big) \Big( \frac{\Delta}{2} - 2 \Big) \\
        &= 2\Big(\frac{\Delta}{2} - 1\Big)^2 + 1.
    \end{align*}

    For $\Delta \geq 5$ odd, note that the primitive moves of $A_{\Delta-1}$ are also primitive moves for $B_\Delta$. So $\bu_{\Delta -1}$ is a primitive move for $B_\Delta$ of degree 
    \[
    2\Big(\frac{\Delta-1}{2} - 1\Big)^2 + 1 = 2\Big(\left\lfloor \frac{\Delta}{2} \right\rfloor -1 \Big)^2 + 1.
    \]
\end{proof}

\begin{corollary}
    The unique minimal Markov basis for $\Lambda(B_\Delta)$ contains a Markov move of degree $4\Big(\left\lfloor \frac{\Delta}{2} \right\rfloor -1 \Big)^2 + 2$.
\end{corollary}

\begin{proof}
    By Theorem \cite[Theorem 7.1]{sturmfels1996grobner}, see also \cite[Proposition~4.3]{aoki2012markov} and \cite[Proposition~5.2]{aoki2012markov}, the unique minimal Markov basis for $\Lambda(B_\Delta)$ consists of the lifts of all primitive Markov moves for $B_\Delta$. The result follows from Theorem \ref{thm:uDeltaDegree} and the observation that the operation of lifting a move for $B_\Delta$ to one for $\Lambda(B_\Delta)$ doubles the degree.
\end{proof}

In fact, computational evidence suggests that the degrees of primitive moves for $B_\Delta$ may grow \emph{exponentially} with $\Delta$. Proving this, and finding a formula for the maximum degree of these primitive moves, is a possible direction for future research.

\section{Lattice basis}\label{Sec:latt_basis}

As we showed in Section \ref{sec:DegreeDiverges}, the complexity of Markov bases diverges quickly as $\Delta$ goes to infinity. A complete description of the Markov moves for arbitrary $\Delta$ seems far-off. For this reason, in this section, we focus instead on finding lattice bases for our matrix $B_\Delta$. Every Markov move is a linear combination of lattice basis elements, and knowledge of this lattice basis will allow us to use the actor-critic reinforcement learning method of \cite{ActorCriticMarkovBases_Petrovic} to explore the fiber of a given JDM.

\begin{definition}
    Let $A \in \mathbb{Z}^{m \times n}$ be an integer matrix, and consider the lattice
\[
L = \ker_{\mathbb{Z}}(A)
  = \{\, \mathbf{x} \in \mathbb{Z}^n \mid A\mathbf{x} = \mathbf{0} \,\}.
\]
A set of integer vectors 
\[
B = \{\mathbf{b}_1, \mathbf{b}_2, \dots, \mathbf{b}_k\} \subset \mathbb{Z}^n
\]
is called a \emph{lattice basis} for $L$ if $B$ is a basis for $L$ as a free $\mathbb{Z}$-module.
\end{definition}

Our goal is to find a lattice basis for the integer kernel of $B_\Delta$. In Proposition 4.12 of \cite{roost2024exploring}, the authors give a formula for the rank of $B_\Delta$ from which we can determine the dimension of $\ker_\mathbb{Z}(B_\Delta)$.

\begin{proposition}[\cite{roost2024exploring}, Proposition 4.12]\label{prop: rank of Adelta}\label{prop:ADeltaRank} For each $\Delta \geq 3$
\[
\rnk(\Lambda(B_{\Delta}))= \binom{\Delta+1}{2}+3\Delta-3 =d-2.
\]
\end{proposition}

Recall that $B_\Delta$ is a $(3\Delta - 1) \times \binom{\Delta+1}{2}$ integer matrix and $\Lambda(B_\Delta)$ is its Lawrence lifting. Since $\Lambda(B_\Delta)$ is block lower-triangular, we see that
\[
\rnk(\Lambda(B_\Delta)) = \binom{\Delta+1}{2} + \rnk(B_\Delta) = \binom{\Delta+1}{2} + 3\Delta - 3.
\]
Thus the rank of $B_\Delta$ is $3\Delta - 3$. 
So to find a lattice basis, it is enough to exhibit ${\Delta +1 \choose  2}-3\Delta+3$ linearly independent integer vectors in the kernel of $B_\Delta$. These can then be extended to a lattice basis for $\Lambda(B_\Delta)$ via the map described in Equation \ref{eq: move lifting}. 

In fact, we will find ${\Delta +1 \choose  2}-3\Delta+3$ linearly independent degree 3 moves.
To find these Markov moves, we first characterize \emph{all} degree 3 moves for every value of $\Delta$. Then we will select a linearly independent subset of them as a lattice basis. We require the following observation that the minimal degree of a move in $\ker(B_\Delta)$ is $3$.

\begin{proposition}\label{prop: no degree 2 moves}
    There are no Markov moves of degree $1$ or $2$ in $\ker(B_\Delta)$.
\end{proposition}

\begin{proof}
    Since no two columns of $B_\Delta$ are equal, there are no degree $1$ moves in $\ker(B_\Delta)$.
    Suppose that 
    \[
    \be_{ij} + \be_{k\ell} - \be{ik} - \be{j \ell}
    \]
    belongs to $\ker(B_\Delta)$. Note that by construction, any move in $\ker(B_\Delta)$ must have this form and must satisfy $i + j = i+k$ and $k + \ell = j + \ell$. But this implies that $j = k$ and that the above Markov move is a the zero vector.
\end{proof}

\begin{proposition}\label{Prop:Degree3Moves}
    Let $i,j,k,\ell$ be integers such that $1 \leq i < j < k < \ell \leq \Delta$. Let $a = i + \ell - j$ and let $b = i + k - j$. Then the vector 
\begin{equation}\label{eq:MarkovDeg3Moves}
\mathbf{b}_{ijk\ell}=\mathbf{e}_{i\ell}+\mathbf{e}_{ak}+\mathbf{e}_{jb}-\mathbf{e}_{ik}-\mathbf{e}_{ja}-\mathbf{e}_{b\ell}
\end{equation}
belongs to $\ker_\ZZ(B_\Delta)$. Moreover, all degree 3 Markov moves for $B_\Delta$ are of this form and all of these are distinct.
\end{proposition}

\begin{proof}
    Each index $i, j, k, \ell, a, b$ appears exactly once with positive and negative coefficient. Moreover, we have
    \begin{eqnarray*}
        i + \ell &= j + i + \ell - j = j +a \\
        a + k &= i + \ell - j + k = b + \ell \\
        j + b &= j + i + k - j = i + k.
    \end{eqnarray*}
    So by construction, $\bb_{ijk\ell} \in \ker_\ZZ(B_\Delta)$. The fact that $i,j,k,\ell$ are all distinct guarantees that $\bb_{ijk\ell}$ is not the zero vector. Indeed, the index $i \ell$ appears with positive coefficient but not with negative coefficient since $k < \ell$ and $b > i$. By Proposition \ref{prop: no degree 2 moves}, $B_\Delta$ has no degree 1 or 2 Markov moves; so we conclude that $\bb_{ijk\ell}$ has degree 3.

    Next, we claim that all degree $3$ moves are of this form. Let $\bb \in \ker_\ZZ(B_\Delta)$ be a degree 3 move and suppose that its positive support $\bb^+ = \be_{uv} + \be_{wx} + \be_{yz}$ where $u$ is minimal among these indices. Then since each index $u,v,w,x,y,z$ appears with the same multiplicity in its negative support $\bb^-$ and no pair of indices appears in both $\bb^+$ and $\bb^-$, we have without loss of generality $\bb^- = \be_{uw} + \be_{vy} + \be_{xz}.$ Indeed, if the index $u$ is paired with one of $w$ or $x$, then $v$ must be paired with one of $y$ or $z$ since otherwise $\be_{yz}$ will appear in both $\bb^+$ and $\bb^-$. After deciding the indices to pair with $u$ and $v$, the last pair is determined. We must also have
    \begin{align*}
        u + v &= x +z, \\
        w+x & = v + y, \text{ and} \\
        y + z  & = u + w.
    \end{align*}
    This is the only way of assigning equalities so that no index appears on both sides of an equality.
    Without loss of generality, we may assume $v < w$; indeed, if not, then negating $\bb$ gives a scalar multiple of this lattice basis element of the desired form. 
    
    We claim that $u < z < v$. If $z < u$, then this contradicts that $u$ was chosen to be minimal among the indices. If $z = u$, then this would imply that $ u + \alpha = u+\beta$ for some $\alpha = v, y$ and $\beta = w, x$. But then one of $v, y$ and one of $w, x$ are equal to each other, which contradicts that no pair of indices appears in both $\bb^+$ and $\bb^-$. Similar reasoning shows that $ z \neq v.$ Finally, if $z > v$, then $u + v - z = x < u$, which contradicts the minimality of $u$. So we conclude that $u < z < v < w$ and this move is $\bb_{uzvw}$. 

    Finally, each of the moves $\bb_{ijk\ell}$ are distinct. Indeed, suppose $\bb_{ijk\ell} = \bb_{i'j'k'\ell'}$. Then since $i$ is the minimal index that appears in either move, we must have $i = i'$. Moreover, $\ell$ and $k$ are the indices that are paired with $i$ and $\ell > k$. So we must have $k = k'$ and $\ell = \ell'$. Finally $j$ is the unique index that appears in both terms of $\bb_{ijk\ell}$ without any of $i,k,$ or $\ell$. So we must have $j = j'$.
\end{proof}

Our goal is to find a subset of the degree 3 moves described in Proposition \ref{Prop:Degree3Moves} that form a lattice basis for $B_\Delta$. Note that the dimension of the kernel of $B_\Delta$ is 
$
{\Delta +1 \choose  2}-3\Delta+3 = {\Delta - 2 \choose 2}$. Let $M$ be the $\binom{\Delta +1}{2} \times \binom{\Delta}{4}$ matrix whose rows are indexed by pairs $uv$ with $1 \leq u \leq v \leq \Delta$ and whose columns are the vectors $\bb_{ijk\ell}$ for $1 \leq i < j < k < \ell \leq \Delta$. We will find a square $\binom{\Delta -2}{2} \times \binom{\Delta -2}{2}$ submatrix of $M$ that is upper-triangular with nonzero entries on its diagonal. Its rows will be indexed by
\[
R=\{xy \ | \ 2<x \leq y<\Delta\},
\]
which we note has cardinality $\binom{\Delta-2}{2}$.
We select the columns of this square submatrix to be indexed by 
\[
I=\{ijk\ell\ | \ i=1, j=y-x+2, \ k=y,\ \ell=y+1 \text{ for each } xy\in R \  \}.
\]
Note that for each $ijk\ell \in I$, we indeed have $i < j < k < \ell$ since $2 < x \leq y$. So $\bb_{ijk\ell}$ is a column of $M$.

\begin{example}\label{ex:LatticeBasis}
    Let $\Delta = 5$. In this case, $R = \{33, 34, 44\}$ and $I = \{1234, 1345, 1245\}$. Note that $\{\bb_{1234}, \bb_{1345}, \bb_{1245} \}$ does indeed span the kernel of $B_5$.
\end{example}

\begin{theorem}
    The $R \times I$ submatrix of $M$ is upper-triangular with nonzero diagonal entries. In particular,
    \[
    \{\bb_{ijk\ell} \mid ijk\ell \in I\}
    \]
    is a lattice basis for $B_\Delta$.
\end{theorem}

\begin{proof}
    We consider the rows of $M$ to be lexicographically ordered. Let $xy \in R$ and let $i = 1$, $j = y -x +2$, $k=y$ and $\ell = y+1$. We wish to find the last entry in the support of $\bb_{ijk\ell}$. Using the notation of Proposition \ref{Prop:Degree3Moves}, we have $a = i + \ell - j = x$ and $b = i + k - j = x-1$. The support of $\bb_{ijk\ell}$, with indices written as sets since their order is not known in all cases, is
    \begin{equation}\label{eq:Lexicographic}
    \Big\{
    \{1, y \}, \{1, y+1\}, \{x-1, y+1\}, \{y - x +2, x\}, \{y-x +2, x-1\} , \{ x,y\}
    \Big\}.
    \end{equation}
    We claim that $\{x,y\}$ is lexicographically last among these. Since $2 < x \leq y$, it is clear that $\{1,y\}, \{1,y+1\}$ and $\{x-1,y+1\}$ are lexicographically before $\{x,y\}$. If $y-x+2 \leq x$, then $\{y - x +2, x\}$ and $ \{y-x +2, x-1\} $ are both lexicographically before $\{x ,y\}$. Moreover, if $y - x +2 > x$, then since $x > 2$, we have $y -x +2 <y$. So in this case, the sets $\{x, y-x+2\}$ and $\{x-1, y-x+2\}$ are still lexicographically before $\{x,y\}$. 

    Hence the last nonzero entry of $\bb_{ijk\ell}$ is indexed by $xy$ and the $R \times I$ submatrix of $M$ is upper-triangular with nonzero diagonal. So $\{\bb_{ijk\ell} \mid ijk\ell \in I\}$ is a collection of $\binom{\Delta - 2}{2}$ elements of $\ker_\ZZ(B_\Delta)$. Since the nullity of $B_\Delta$ is $\binom{\Delta - 2}{2}$, these form a lattice basis for $\ker_\ZZ(B_\Delta)$. Extending them by
    \[
    \overline{\bb}_{ijk\ell} = \begin{pmatrix}
        \bb_{ijk\ell} \\
        -\bb_{ijk\ell}
    \end{pmatrix}
    \]
    yields a lattice basis for $\ker_\ZZ(B_\Delta)$.
\end{proof}

\begin{example}
    Again, let $\Delta = 5$ so that $I$ is as in Example \ref{ex:LatticeBasis}. The indices in the support of $\bb_{1234}, \bb_{1345}$ and $\bb_{1245}$ are $\{13, 14, 24, 23, 22, \mathbf{33}\}, \{14, 15, 25, 33, 23, \mathbf{34} \}$ and $\{14, 15, 35, 24, 23, \mathbf{44} \}$, respectively, written in the order prescribed in Equation \ref{eq:Lexicographic}. The lexicographically last entry in each of these is highlighted, and we note that in each case it is equal to the given $xy$.
\end{example}

\section{Exploration of Joint Degree Matrix Fibers via Reinforcement Learning}\label{Section:ReinforcemengtLearning}

As discussed in the previous sections, a minimal Markov basis $\mathcal{M}$ for an integer matrix $A \in \mathbb{Z}^{d \times n}$ ensures the connectivity of every fiber $\mathcal{F}_A(u)$ for all vectors $u \in \mathbb{N}^n$. However, for a fixed vector $u$ and its corresponding fiber $\mathcal{F}_A(u)$, it is typically unnecessary to use all elements of the full Markov basis to achieve connectivity \cite{dobra2012dynamic}. To the best of our knowledge, there is no exact method for determining a fiber-specific set of moves more efficiently than computing the complete Markov basis. Nevertheless, in \cite{ActorCriticMarkovBases_Petrovic}, the problem is reformulated as a Markov decision process, and a reinforcement learning (RL) approach is proposed to learn ``good'' moves for exploring a specific fiber.

We apply this RL method to the problem of exploring Joint Degree Matrices (JDMs) corresponding to given degree and curvature sequences. This approach enables us to compute moves for fibers generated by graphs with a maximum degree higher than $\Delta = 8$, which is the limit at which the algebraic software \texttt{4ti2} fails to terminate within several days. While replacing exact algebraic techniques with a stochastic RL algorithm reduces theoretical certainty, a rigorous convergence analysis for this method is provided in \cite[Section~4]{ActorCriticMarkovBases_Petrovic}. One limitation of this method is the need in practice to bound a priori the coefficients in the linear combinations of elements of lattice bases to sample an admissible move (that is, an element of a Markov basis), see \cite[Page 39]{ActorCriticMarkovBases_Petrovic}. However, for a given fiber, we have a natural bound for the maximum integer coefficients in the linear combination in our case. In fact, we know that all the graphs with the same degree sequence have the same number of edges, as $|E(G)|=\tfrac{1}{2}\sum^{\Delta}_{a=1} a|V_a|=\tfrac{1}{2}\left(\sum^{\Delta}_{a,b=1} J_{ab}+\sum^{\Delta}_{a=1}J_{aa}\right) \geq \sum_{1\leq a
\leq b\leq \Delta} J_{ab} $, where the symmetric $\Delta\times \Delta$ matrix $J$ is the JDM of the graph $G$.  Therefore, the degree of the Markov moves applicable to the fiber of $J$ is trivially upper bounded by twice the number of edges corresponding to the fiber of $J$.

We evaluate the reinforcement learning algorithm on various graphs with different maximum degrees and JDMs. Specifically, we test the algorithm on realizations of Erd\H{o}s--R\'{e}nyi and Barab\'{a}si--Albert random graphs \cite[see Sections 5 and 8]{van2017random}, as well as the Karate Club graph \cite{zachary1977information}, a standard benchmark in network science. All graph realizations, including the Karate Club graph, are generated using the \texttt{networkx} package. For random graphs we consider a specific realization fixing \texttt{seed=31} to ensure deterministic and reproducible outputs. Additionally, we test graphs with ``dense'' JDMs (those with no zero entries), such as the graph associated with the following JDM:

\setcounter{MaxMatrixCols}{20}

\begin{equation}\label{eq:C14Matrix}
\begingroup
\renewcommand{\arraystretch}{1.2} 
\setlength{\arraycolsep}{2.5pt}    
\small                            
\mathcal{C}_{15}=\begin{pmatrix}
1 & 1 & 1 & 1 & 1 & 1 & 1 & 1 & 1 & 1 & 1 & 1 & 1 & 1 & 1 \\
1 & 1 & 1 & 1 & 1 & 1 & 1 & 1 & 1 & 1 & 1 & 1 & 1 & 1 & 1 \\
1 & 1 & 2 & 1 & 1 & 1 & 1 & 1 & 1 & 1 & 1 & 1 & 1 & 1 & 1 \\
1 & 1 & 1 & 1 & 1 & 1 & 1 & 1 & 1 & 1 & 1 & 1 & 1 & 1 & 1 \\
1 & 1 & 1 & 1 & 3 & 1 & 1 & 1 & 1 & 1 & 1 & 1 & 1 & 1 & 1 \\
1 & 1 & 1 & 1 & 1 & 2 & 1 & 1 & 1 & 1 & 1 & 1 & 1 & 1 & 1 \\
1 & 1 & 1 & 1 & 1 & 1 & 14 & 1 & 1 & 1 & 1 & 1 & 1 & 1 & 1 \\
1 & 1 & 1 & 1 & 1 & 1 & 1 & 1 & 1 & 1 & 1 & 1 & 1 & 1 & 1 \\
1 & 1 & 1 & 1 & 1 & 1 & 1 & 1 & 38 & 1 & 1 & 1 & 1 & 1 & 1 \\
1 & 1 & 1 & 1 & 1 & 1 & 1 & 1 & 1 & 43 & 1 & 1 & 1 & 1 & 1 \\
1 & 1 & 1 & 1 & 1 & 1 & 1 & 1 & 1 & 1 & 59 & 1 & 1 & 1 & 1 \\
1 & 1 & 1 & 1 & 1 & 1 & 1 & 1 & 1 & 1 & 1 & 65 & 1 & 1 & 1 \\
1 & 1 & 1 & 1 & 1 & 1 & 1 & 1 & 1 & 1 & 1 & 1 & 84 & 1 & 1 \\
1 & 1 & 1 & 1 & 1 & 1 & 1 & 1 & 1 & 1 & 1 & 1 & 1 & 91 & 1 \\
1 & 1 & 1 & 1 & 1 & 1 & 1 & 1 & 1 & 1 & 1 & 1 & 1 & 1 & 113
\end{pmatrix}.
\endgroup
\end{equation}

We remark that the actor-critic algorithm is inherently stochastic and may not identify all possible moves in a single execution. Consequently, the number of sampled states may vary across multiple runs, and the number of moves discovered might be smaller than the true cardinality of the fiber. However, the states identified by the algorithm provide a guaranteed lower bound for the cardinality of the elements contained in the fiber of a given JDM. Our findings are summarized in Table \ref{tab:my_table}.

 \begin{table}[h]
    \centering
    \begin{tabular}{|c|c|c|}
        \hline
        \textbf{Graph}& \textbf{Sampled states} & \textbf{Maximum degree (or JDM size)} \\ \hline
        G(50, 0.08)& 4& 7 \\ \hline
        G(100,0.02)& 105& 5 \\ \hline
        G(250,0.02)& 29 & 12\\ \hline
        G(1000,0.02)& 1 & 34\\ \hline
       $\mathcal{C}_{15}$ (Equation \ref{eq:C14Matrix})& 1749& 15 \\ \hline
        Barabasi-Albert(30,2)& 1 & 12 \\ \hline
        Barabasi-Albert(100,3)& 1 & 39 \\ \hline
        Karate club graph& 1 & 17 \\ \hline
    \end{tabular}
    \caption{Table of explored/sampled states for different graphs}
    \label{tab:my_table}
\end{table}

\subsection*{Data}

The code for the Reinforcement Learning algorithm adapted to our problem of JDMs exploration is available at the following link \hyperlink{https://github.com/GiulioZucal/Fiber-Sampling-Using-Reinforcement-Learning-DiscreteCurvature}{https://github.com/GiulioZucal/Fiber-Sampling-Using-Reinforcement-Learning-DiscreteCurvature} (Zenodo DOI: 10.5281/zenodo.21722444).

\subsection*{Acknowledgements}
This project began at the 1st MPI (Dresden + Leipzig) Day hosted by the Max Planck Institute for Mathematics in the Sciences. We thank the organizers, Türkü Özlüm Çelik and Irem Portakal, for their hospitality and for organizing the event. We thank Karel Devriendt for several helpful discussions.

\bibliographystyle{abbrv}
\bibliography{bibliography.bib}
\end{document}